\documentclass[12pt,sumlimits,intlimits]{amsart}
\usepackage{url,cite,hyperref,amsmath,amsthm,amssymb,mathtools}
\usepackage[margin=1.25in]{geometry}
\usepackage{xcolor}
\usepackage{comment}
\usepackage{mathbbol}

\newtheorem{theorem}{Theorem}[section]
\newtheorem{proposition}[theorem]{Proposition}
\newtheorem{corollary}[theorem]{Corollary}
\newtheorem{lemma}[theorem]{Lemma}
\theoremstyle{definition}
\newtheorem{definition}[theorem]{Definition}
\newtheorem{remark}[theorem]{Remark}
\newtheorem{example}[theorem]{Example}

\newtheorem{question}[theorem]{Question}

\newcommand{\la}{\langle}
\newcommand{\ra}{\rangle}

{\large }

\newcommand{\ZH}{\mathbb{Z}[\frac{1}{2}]}
\newcommand{\HZH}{\mathbb{H}_3(\ZH)}
\newcommand{\dZH}{\widehat{\ZH}}
\newcommand{\Z}{\mathbb{Z}}

\newcommand{\C}{\mathbb{C}}
\newcommand{\N}{\mathbb{N}}
\newcommand{\R}{\mathbb{R}}
\newcommand{\Q}{\mathbb{Q}}

\newcommand{\T}{\mathbb{T}}

\numberwithin{equation}{section}
\title{On amenable Hilbert-Schmidt stable groups}
\author{Caleb Eckhardt, Tatiana Shulman}
\address{Department of Mathematics, Miami University, Oxford, Ohio}
\email{eckharc@miamioh.edu}
\address{University of Gothenburg, Sweden.}
\email{tatshu@chalmers.se}
\thanks{T.S.\ is partially supported by a grant from the Swedish Research Council and by the Polish
National Science Centre grant under the contract number 2019/34/E/ST1/00178}

\begin{document}
\maketitle
\begin{abstract} We examine Hilbert-Schmidt stability (HS-stability) of discrete amenable groups from several angles.
We give a short, elementary proof that finitely generated nilpotent groups are HS-stable. We investigate the permanence of HS-stability under central quotients by showing HS-stability is preserved by finite central quotients, but is not preserved in general. We give a characterization of HS-stability for semidirect products $G\rtimes_\gamma \Z$  with $G$ abelian. We use it to construct the first example of a finitely generated amenable HS-stable group which is not permutation stable. Finally, it is proved that for amenable groups flexible HS-stability is equivalent to HS-stability, and very flexible HS stability is equivalent to maximal almost periodicity.  There is some overlap of our work with the very recent and very nice preprint \cite{Levit22} of Levit and Vigdorovich. We detail this overlap in the introduction.  Where our work overlaps it appears that we take different approaches to the proofs and we feel the two works compliment each other.

\end{abstract}
\section{Introduction}
A famous result of Voiculescu states that a pair of almost commuting unitary matrices need not be close in operator norm to a pair of actually commuting unitary matrices \cite{Voiculescu}. This can be rephrased by saying that an approximate representation of the group $\mathbb Z^2$  is not necessarily close to an actual representation.  In this sense $\mathbb Z^2$ is {\it not stable}. Similarly one defines stability of an arbitrary discrete group with respect to a given norm by declaring

\medskip

{\it  A group $G$ is stable with respect to a given norm if each of its approximate representations is close to a representation.}\footnote{See Section \ref{sec:Preliminaries} for details}

\medskip

Thus $\mathbb Z^2$ is not stable with respect to the operator norm and in fact operator norm stability is rare (\cite{ESS}, \cite{Dadarlat}) although there are some non-trivial examples (\cite{ESS}).

On the other hand, the normalized Hilbert-Schmidt norm is much more forgiving, and there are many examples of Hilbert-Schmidt stable groups. In the class of amenable groups there were not many examples (\cite{Hadwin18}, \cite{Glebsky}, \cite{HadwinLi}) until the recent preprint of Levit and Vigdorovich \cite{Levit22}, showed that all  finitely generated virtually nilpotent groups and many metabelian groups are HS-stable.

Non-amenable examples of HS-stable groups include certain graph product groups \cite{Atkinson},  one-relator
groups with non-trivial center \cite{Hadwin18}, and virtually free groups \cite{Gerasimova}.  We also note that HS-stability passes to free products and direct products with  abelian \cite[Th.1]{Hadwin18} or, more generally, amenable HS-stable groups \cite[Cor. D]{Ioana21a}.  Locating examples/non-examples of HS-stable groups  and exploring behavior under group theoretic constructions is a topic of intensive study, partly motivated by connection with the search of a non-hyperlinear group and quantum information aspects of Connes embedding problem (see e.g. \cite{Chiffre}, \cite{Thom}, \cite{CEP}, \cite{Salle}). In particular, to find a non-hyperlinear group it is sufficient to find an HS-stable non maximally periodic group.

Switching to non-examples of HS-stable groups, we recall some of the known obstructions to HS-stability.  In the non-amenable case, besides the ``ad-hoc" $F_2\times F_2$ (\cite{Ioana21}), all other non-examples use property T or property $\tau$  groups \cite{Becker20}, \cite{Ioana20}.   For amenable groups, the situation is quite different. First off, HS-stability of amenable groups admits the following useful characterization
\medskip

{\it An amenable group is HS-stable if and only if each of its traces is a pointwise limit of normalized traces of
finite dimensional representations \cite{Hadwin18}.}

\medskip
By considering the trace $\delta_e(g) = \begin{cases} 1, g=e \\ 0, g\neq e \end{cases}$ this characterization   more-or-less immediately yields the following obstruction to HS-stability:

\medskip

{\it Every amenable HS-stable group is maximally almost periodic.}\footnote{Alternatively, and even more naturally, one may use  the well-known hyperlinearity of amenable groups and the definition of HS-stability in terms of liftings, see section \ref{sec:Preliminaries}.}

\medskip

\noindent
With the above viewpoint for amenable groups,  one may directly view HS-stability as a very strong form of maximal almost periodicity.
Indeed, an amenable group is MAP if and only if the trace $\delta_e$  is a pointwise  limit of normalized traces of
finite dimensional representations  (\cite[Prop.3]{Hadwin18}).  On the other hand HS-stability requires \emph{every} trace to be approximated through matrices--not just $\delta_e.$  However it is unknown whether this difference is merely a formality--in other words if the ``obvious" obstruction to HS-stability is the \emph{only} obstruction. Explicitly, we ask
\begin{question}[MAP implies HS-stable?]\label{ques:MAP/HSstable} Are there any MAP amenable groups that are not HS-stable?
\end{question}

In one of the simplest classes of groups we obtain a dynamical reformulation of Question \ref{ques:MAP/HSstable} that is also unknown and worth highlighting.  Let $G$ be an abelian group and $\gamma$ an automorphism of $G.$  By Proposition \ref{prop:Mapchar},  $G\rtimes_\gamma\Z$ is MAP precisely when the periodic points of the dual action  $\hat\gamma$ on $\hat G$ are dense. On the other hand by Corollary \ref{cor:densinvmeasure}, $G\rtimes_\gamma\Z$ is HS-stable when the $\hat \gamma$-invariant probability measures on $\hat G$ with finite support are weak*-dense in all $\hat \gamma$-invariant probability measures on $\hat G.$ It is trivial in this case that HS-stable implies MAP and we may recast Question \ref{ques:MAP/HSstable} as

\begin{question}[MAP implies HS-stable? Simplified version] Is there any abelian group $G$ and an automorphism $\gamma$ of $G$ such that the $\hat\gamma$ periodic points are dense but the $\hat \gamma$-invariant probability measures on $\hat G$ with finite support are \emph{not} weak*-dense in the set of all  $\hat \gamma$-invariant probability measures?
\end{question}

We now turn to a detailed description of the contents of this paper.  In Section \ref{sec:Preliminaries} we begin with the necessary group theoretic and operator algebraic preliminaries used throughout the paper.
In Section \ref{sec:nilpotentcase} we expand the known class of HS-stable groups with

{\bf Theorem \ref{thm:nilapprox}.}  Every finitely generated nilpotent group is HS-stable.

\bigskip
 Theorem \ref{thm:nilapprox} was recently and independently obtained by Levit and Vigdorovich \cite{Levit22}.  In fact they show HS-stability of all finitely generated \emph{virtually} nilpotent groups.  Our result is less general than their's but our proof is brief and elementary so we include it.

 In Section \ref{sec:MAP} we establish several results about MAP groups and establish two useful obstructions to maximal almost periodicity.   In particular the obstructions from Section \ref{sec:MAP} lead to many easy examples of nilpotent non-MAP groups showing that the finitely generated assumption of Theorem \ref{thm:nilapprox} is essential. It is possible some of the results of Section \ref{sec:MAP} are known but we were not able to find them in the literature.

Next we investigate central extensions and the examples they lead to in Section \ref{sec:examples}. We start with a simple observation.

\bigskip

{\bf Proposition \ref{prop:stableext}.} Suppose we have an extension $1 \to N \to G \to H \to 1$, where $N$ is finite and central and $G$ is an HS-stable group. Then $H$ is HS-stable.

\bigskip

The assumptions of Proposition \ref{prop:stableext} cannot be pushed any further. Namely

\bigskip

{\bf Corollary \ref{2-stepNilpotent}.} For the central extension
\begin{equation*}
0\rightarrow \Z\rightarrow \mathbb{H}_3(\mathbb{Z}[\tfrac{1}{2}])\rightarrow \mathbb{H}_3(\mathbb{Z}[\tfrac{1}{2}])/\Z\rightarrow 0
\end{equation*}
the group $\HZH$ is HS-stable, but the quotient $\HZH/\Z$ is not MAP and therefore not HS-stable.  Hence Proposition \ref{prop:stableext} can not be extended beyond the finite case.

The impetus for studying central extensions was to use the ideas of \cite{Becker19} to try and locate an amenable, HS-stable group that is not permutation stable. Along the way to achieving this we obtained the following

\bigskip

{\bf Corollary \ref{cor:densinvmeasure}.} Let $G$ be a countable abelian group and $\gamma$ an automorphism.  Then $G\rtimes_\gamma\Z$ is HS-stable if and only if the $\hat\gamma$-invariant measures on $\hat G$ with finite support are weak*-dense in the set of all $\hat\gamma$-invariant measures on $\hat G.$

\bigskip
Levit and Vigdorovich recently and independently obtained a very similar result \cite[Theorem C]{Levit22}.  We feel the proofs are quite different\footnote{In particular our proof heavily uses operator algebras} and each approach adds new insight into HS-stability.
We use the ideas leading up to Corollary \ref{cor:densinvmeasure}  to give the first example of a finitely generated amenable HS-stable non-permutation stable group. Permutation stability requires approximate homomorphisms to permutation groups to be close to homomorphisms, with respect to the normalized Hamming distance.  It plays the same role in the context of sofic conjecture as HS-stability -- in the context of Connes embedding problem for groups. Namely to disprove sofic conjecture, it is sufficient to find a permutation stable non - residually finite  group. Ioana \cite{Ioana20a}  showed that $F_2\times \mathbb Z$ is not permutation stable while it is HS-stable by \cite{Hadwin18}, hence permutation and HS-stability were known to be different in general. However it was not known whether these two versions of stability are equivalent in the class of \emph{amenable} groups. We settle this problem with
\bigskip

{\bf Theorem \ref{HS-stableNotPermutationStable}.} Define $\beta:\HZH\to\HZH$ by
\begin{equation*}
\beta\left( \left[ \begin{array}{ccc} 1 & x & z\\ 0 & 1 & y\\ 0 & 0 & 1 \end{array}  \right]  \right)=\left[ \begin{array}{ccc} 1 & 2x & z\\ 0 & 1 & \frac{y}{2}\\ 0 & 0 & 1 \end{array}  \right]
\end{equation*}
Then the group $\HZH\rtimes_\beta\Z$ is (finitely generated) HS-stable. Moreover by the proof of \cite[Corollary 8.7]{Becker19},  $\HZH\rtimes_\beta\Z$ is not permutation stable.

\bigskip

We finish with Section \ref{sec:flexiblestability} characterizing  flexible HS-stability and very flexible HS-stability of amenable groups.    The notions of (very) flexible stability were introduced by Becker and Lyubotzky \cite{Becker20} and they require any approximate representation to only be close to a
 ``(large)
corner" of a representation (see section \ref{sec:flexiblestability} for details). Again, if  a Connes-embeddable countable group  is (very)  flexibly HS-stable, then it must be MAP.  Examples of non-flexibly stable groups can be found in \cite{Ioana20}, \cite{Ioana21}.
We continue our theme of highlighting the (possibly non-existent) line between HS-stable and MAP amenable groups with the following
\bigskip

{\bf Proposition \ref{FlexiblyStable}.} Let $G$ be amenable. Then $G$ is flexibly HS-stable iff it is HS-stable.

\bigskip

{\bf Corollary \ref{VeryFlexiblyStable}.} Let $G$ be amenable. Then $G$ is very flexibly HS-stable iff $G$ is MAP.

\bigskip
In light of these results, Question \ref{ques:MAP/HSstable} is equivalent to asking if there are any very flexible HS-stable amenable groups that are not flexibly HS-stable.

\section{Preliminaries}\label{sec:Preliminaries}
\subsection{Group preliminaries}
We record some group theory definitions and results that will be used frequently in this work.
Let $G$ be a group. We denote by $e\in G$ the trivial element.  We write $N\lhd_f G$ to mean that $N$ is a normal subgroup of finite index. The group $G$ is \textbf{residually finite} if
\begin{equation*}
\bigcap_{N\lhd_f G} N=\{ e \}.
\end{equation*}
Equivalently $G$ is residually finite if for every $x\in G$ there is a finite group $F$ and a homomorphism $\pi:G\to F$ such that $\pi(x)\neq e$.

Let $\mathcal{U}(n)$ be the unitary group on $n\times n$ matrices.  A group $G$ is called \textbf{maximally almost periodic} (abbreviated MAP) if for every $x\in G$ there is an $n\in\N$ and a homomorphism $\pi:G\to \mathcal{U}(n)$ such that $\pi(x)\neq e.$
By the Peter-Weyl theorem a group is MAP if and only if it embeds into a compact group.

Clearly every residually finite group is MAP but in general the converse is not true, for example $\Q$ is MAP but not residually finite. On the other hand, Mal'cev showed \cite{Malcev40} that every \emph{finitely generated} linear group is residually finite, in particular for finitely generated groups the notions of MAP and residually finite coincide.

A \textbf{trace} on a group $G$ is a positive definite function $\tau:G\rightarrow\C$ that is constant on conjugacy classes and satisfies $\tau(e)=1.$ A \textbf{character} of $G$ is an extreme point in the set of all traces on $G$.

Throughout this paper by a {\bf representation} we mean a unitary representation of a group, that is a homomorphism either to $\mathcal U(n)$, for some $n\in \mathbb N$, or to the group of unitary operators on a Hilbert space. For an integer $n\geq0$ we write $M_n$ for the $n\times n$ matrices with complex coefficients and write $\text{tr}_n:M_n\rightarrow \C$ for the trace that maps the identity matrix to 1. For an element $T\in M_n$ define the \textbf{normalized Hilbert-Schmidt norm} as $\| T \|_2=(\text{tr}_n(T^*T))^{1/2}.$

 \medskip

An {\bf approximate representation} of a group $G$ is a sequence of unital maps $\phi_n: G \to \mathcal U(k_n)$, $n\in \mathbb N$, such that for all $g_1,g_2\in G$ we have
\begin{equation*}
\lim_{n\rightarrow\infty}\|\phi_n(g_1g_2) - \phi_n(g_1)\phi_n(g_2)\|_2 =0.
\end{equation*}
A group $G$ is {\bf Hilbert-Schmidt stable} (HS-stable) if for any approximate representation $\phi_n: G \to \mathcal U(k_n)$, $n\in \mathbb N$, there is a sequence of representations $\pi_n: G \to \mathcal U(k_n)$, $n\in \mathbb N$, such that for all $g\in G$
\begin{equation*}
\lim_{n\rightarrow\infty}\|\phi_n(g) - \pi_n(g)\|_2 =0.
\end{equation*}
\medskip

Equivalently, $G$ is HS-stable if for any non-trivial ultrafilter $\omega$ on $\mathbb N$ and any  sequence $k_n$, any homomorphism from $G$ to the ultraproduct  group
\begin{equation*}
\prod^{\omega} (\mathcal U(k_n), \|\|_2) = \prod \mathcal U(k_n) / \{(V_n)_{n\in \mathbb N} \;|\; \|V_n - \mathbb 1\|_2 \to_{\omega} 0\} \end{equation*}
lifts to a homomorphism from $G$ to the product $\prod \mathcal U(k_n)$ (\cite{Hadwin18}).

A group is {\bf hyperlinear} if it embeds to the ultraproduct group $\prod^{\omega} (\mathcal U(k_n), \|\|_2)$, for some choice of $\omega$ and a sequence $k_n$.  The Connes Embedding Conjecture for groups states that any  discrete group is hyperlinear.  Although the Connes Embedding Problem for $II_1$-factors was recently disproved in \cite{CEP}, the Connes Embedding Conjecture for groups is open. Summarizing the above paragraphs every hyperlinear, HS-stable group is MAP. 
It is well known and easy to prove that amenable groups are hyperlinear.  Consequently an amenable HS-stable group is MAP.

Finally we use the following very useful reformulation of HS-stability for amenable groups throughout the paper

\begin{theorem}[\cite{Hadwin18} Theorem 4, Lemma 1] \label{thm:HSchar}
Let $G$ be an amenable group.  Then the following are equivalent
\begin{enumerate}
\item  $G$ is HS-stable 
\item For every trace $\tau$ on $G$ there is a sequence of finite $k_n$-dimensional representations $\pi_n$ of $G$ such that $\tau(g) = \lim \text{tr}_{k_n}\circ\pi_n(g)$ for all $g\in G$
\item For every character $\tau$ on $G$ there is a sequence of finite $k_n$-dimensional representations $\pi_n$ of $G$ such that $\tau(g) = \lim \text{tr}_{k_n}\circ\pi_n(g)$ for all $g\in G$
\end{enumerate}
\end{theorem}

\subsection{Operator algebraic preliminaries}\label{sec:opalgpre}
We recall some frequently used facts about operator algebras and fix our notation at the same time.  We refer the reader to Blackadar's text \cite{Blackadar06} for general information on operator algebras including almost everything we mention here.

Let $A$ be a unital C*-algebra. We denote its unitary group by $\mathcal U(A)$ and the unit by $1_A.$
A \textbf{trace} $\tau$ on $A$ is a positive linear functional $\tau:A\rightarrow \C$ such that $\tau(xy)=\tau(yx)$ for all $x,y\in A$ and $\tau(1_A)=1.$  We write $\pi_\tau$ for the GNS representation associated with $\tau$ and $\pi_\tau(A)''$ is the von Neumann algebra generated by $\pi_\tau(A).$  The set of all traces on $A$ form a convex set.  A trace $\tau$ is extreme if and only if the von Neumann algebra $\pi_\tau(A)''$ is a \textbf{factor}, i.e. if $\pi_\tau(A)''$ has trivial center \cite[Corollary 6.8.6]{Dixmier77}

Let $\gamma$ be an automorphism of $A.$  We write $A\rtimes_\gamma\Z$ for the \textbf{crossed product} \cite[II.10]{Blackadar06}.
Let $u\in A\rtimes_\gamma\Z$ be the \textbf{implementing unitary}, i.e. $uau^*=\gamma(a)$ for all $a\in A.$

If $\tau$ is a $\gamma$-invariant trace on $A$, then it extends to $A\rtimes_\gamma\Z$ by setting $\tau(xu^n)=0$ when $x\in A$ and $n\neq0.$ We call this the \textbf{trivial extension.} For a discrete group $G$ we write $C^*(G)$ for the full group C*-algebra.  All of the groups considered in this paper are amenable so we make no distinction between the full and reduced group C*-algebras and crossed products. We recall \cite[II.10]{Blackadar06} that when $G$ is abelian the Fourier transform gives an isomorphism $C^*(G)\cong C(\hat G)$ and for semidirect products $G\rtimes_\gamma \Z$ we have $C^*(G\rtimes_\gamma \Z)\cong C^*(G)\rtimes_\gamma\Z$ where $\gamma$ is the induced automorphism of $C^*(G).$

The group $G$ sits naturally inside of $C^*(G)$ \cite[II.10]{Blackadar06}.  This provides an affine homeomorphism between traces on a group $G$  and traces on the C*-algebra $C^*(G)$.  In general we make no distinction and use the same symbol $\tau$ when talking about the trace on $G$ or on $C^*(G).$

A  C*-algebra $A$ is $k$-\textbf{homogeneous} if every irreducible representation of $A$ is $k$-dimensional (see \cite[IV.1.4]{Blackadar06} and also Fell's original paper \cite{Fell61}). The primitive ideal space, or spectrum of $A$, is denoted by $\hat A.$  For each ideal $I\in \hat A$ we have $A/I\cong M_k.$  The spectrum $\hat A$ of a $k$-homogeneous C*-algebra is a compact Hausdorff space.  For each trace $\tau$ on $A$ we obtain a Borel probability measure $\nu_\tau$ on $\hat A$ such that
\begin{equation}\label{eq:traceonhomo}
\tau(x)=\int_{\hat A}\text{tr}_k(x+I)d\nu_\tau(I)
\end{equation}
Let $A$ be a $k$-homogeneous C*-algebra and $\gamma$ an automorphism of $A.$  The map $I\mapsto \gamma(I)$ is a homeomorphism of $\hat A$ that we write as $\hat\gamma.$  For any $\gamma$-invariant trace $\tau$ on $A$ we have that $\hat \gamma$ is a measure preserving action on $(\hat A,\nu_\tau).$

Moreover suppose that there is a $\hat \gamma$ periodic point $I$ with period dividing $n.$   Let $\pi_I:A\to A/I\cong M_k$ be the quotient homomorphism. Since $\gamma^n(I)=I$, it follows that $\gamma^n$ descends to an automorphism of $A/I\cong M_k.$   Since all automorphisms of $M_k$ are inner there is a unitary $V\in M_k$ such that
\begin{equation}\label{eq:inneraut}
\pi_I(\gamma^n(x))=V\pi_I(x)V^* \text{ for all }x\in A.
\end{equation}
Let $u\in A\rtimes_\gamma\Z$ be the generating unitary. We define $\pi_{I, n}:A\rtimes_\gamma\Z\to M_{kn}$ by
\begin{equation}\label{eq:halfinduced}
\pi_{I,n}(x) = \bigoplus_{i=0}^{n-1} \pi_I\circ\gamma^i(x), \text{ for  }x\in A\text{ and }\pi_{I,n}(u) =
\left[  \begin{array}{ccccc} 0 & 1 & 0 & \cdots& 0 \\
                                          0 & 0 & 1 & \cdots & 0\\
                                          \vdots & \vdots & \vdots & \ddots&\vdots\\
                                          0 & 0 & 0 & \cdots & 1\\
                                          V & 0 & 0 & \cdots & 0  \end{array}        \right]
\end{equation}
By (\ref{eq:inneraut}) it is clear that $\pi_{I, n}$ determines a representation of $A\rtimes_\gamma\Z.$
Define $\nu_{I,n} = n^{-1}\sum_{i=0}^{n-1}\delta_{\gamma^{i}(I)}.$ Then $\nu_{I,n}$ is clearly a $\hat\gamma$-invariant probability measure on $\hat A.$ Let $\tau$ be the trivial extension to $A\rtimes_\gamma\Z.$  We then have
\begin{equation}\label{eq:tpres}
\tau(xu^i)=\text{tr}_{kn}\circ\pi_{I,n}(x)\text{ for all }x\in A\text{ and for all }-n<i<n.
\end{equation}

\section{HS-stability of finitely generated nilpotent groups}\label{sec:nilpotentcase}
In this section we show that all finitely generated nilpotent groups are HS-stable.  In \cite{Levit22} it was independently shown that all finitely generated  \emph{virtually} nilpotent groups are HS-stable.
Our result is less general than theirs but the proof is straightforward and elementary, so we include it. We first recall some well-known facts about nilpotent groups.
\begin{lemma}\label{lem:approxhomo} Let $H$ be a finitely generated group and $\pi:H\rightarrow \T$ a homomorphism.  Then there is a sequence of homomorphisms $\pi_n:H\rightarrow \T$ such that $\textup{ker}(\pi_n)$ is finite index in $H$ and $\pi_n(h)\rightarrow \pi(h)$ for all $h\in H.$
\end{lemma}
\begin{proof} By replacing $H$ with $H/\textup{ker}(\pi)$ we may assume that $H$ is a finitely generated abelian group. Then $H\cong \Z^d\times T$ where $T$ is the torsion subgroup of $H.$  One builds $\pi_n$ by appropriately approximating $\pi(e_i)$ for a basis $\{ e_1,...,e_d \}$ on the free part of $H$ and setting $\pi_n=\pi$ on $T.$
\end{proof}
\begin{definition} Let $G$ be a group and $H\leq G.$  Let $K_1,K_2,...$ be a sequence of finite index, normal subgroups of $G.$  We say that the sequence $(K_n)$ \textbf{separates} $H$ if for every $x\in G\setminus H$ there is an index $n_0$ so $n\geq n_0$ implies that $xK_n \not\in HK_n.$ If such a sequence exists we say that $H$ is \textbf{separated} in $G.$
\end{definition}
\begin{remark} \label{rem:moresep} Suppose that $(K_n)$ separates a subgroup $H\leq G$ and that $(L_n)$ is any sequence of finite index, normal subgroups of $G.$  Then $(K_n\cap L_n)$ also separates $H.$
 \end{remark}
The following is straightforward and well-known.  One may take an appropriate subsequence of the groups $G^n=\la  g^n:g\in G \ra$ as a separating sequence.
\begin{lemma} \label{thm:nilsep} Every subgroup of a finitely generated nilpotent group is separated.
 \end{lemma}
\begin{lemma}\label{lem:extendgroups} Let $G$ be a finitely generated nilpotent group and $H\leq Z(G)$ a finite index subgroup of the center.   Then there is a normal subgroup $N\trianglelefteq G$ of finite index such that $N\cap Z(G)=H.$
\end{lemma}
\begin{proof} The group $H$ is clearly normal in $G.$ The group $G/H$ is nilpotent and finitely generated hence it is residually finite.  Let $x_0=e,x_1,\ldots,x_n\in Z(G)$ be a complete set of coset representatives for $Z(G)/H.$  Let $K$ be a finite group and $\pi: G/H\rightarrow K$ a homomorphism such that $\pi(x_iH)\neq e$ for $i=1,\ldots,n$.  Let $\sigma:G\rightarrow G/H$ be the quotient map, and set $N=\textup{ker} (\pi\circ\sigma).$ Then $N$ is clearly finite index in $G$  and $H\leq N.$  Finally, let $i\in \{ 1,\ldots,n \}$ and $h\in H.$ Then $\pi(\sigma(x_ih))=\pi(x_iH)\neq e,$ i.e. $x_ih\not\in N.$  Hence $H=N\cap Z(G).$
\end{proof}

\begin{lemma}\label{lem:trivint} Let $G$ be nilpotent and $(L_n)$ a sequence of normal subgroups such that
$\bigcap_{n=1}^\infty (L_n\cap Z(G)) =\{e\}.$ Then $\bigcap_{n=1}^\infty L_n = \{ e \}.$
\end{lemma}
\begin{proof} This is an immediate consequence of the fact that every non-trivial normal subgroup  of a nilpotent group must intersect the center non-trivially.
\end{proof}

For a group $G$ we let $G_f$ denote the (normal) subgroup consisting of all elements with finite conjugacy class.

\begin{lemma} \label{lem:containedinfc} Let $G$ be a finitely generated nilpotent group and $H\leq G_f.$  Then there is a finite set $F\subseteq G$ such that for any $a\in G$ we have
\begin{equation*}
aHa^{-1} \subseteq \bigcup_{b\in F} bHb^{-1}.
\end{equation*}
\end{lemma}
\begin{proof} Since $Z(G)$ has finite index in $G_f$ \cite[Lemma 3]{Baer48},  it follows that $Z(G)\cap H$ has finite index in $H.$  Choose coset representatives $x_1,...,x_k\in H$ for $H/(Z(G)\cap H).$  For each $i$, there is a finite subset $F_i\subseteq G$ such that $\textup{conj}(x_i)=\{ bx_ib^{-1}:b\in F_i \}.$

Let $a\in G$ and $h\in H.$  Then $h=x_iz$ for some $i$ and $z\in Z(G).$  There is a $b\in F_i$ so $ax_ia^{-1}=bx_ib^{-1}.$
Hence $aha^{-1}=bx_izb^{-1}$ so setting $F=\cup_{i=1}^k F_i$ completes the proof.

\end{proof}
\begin{lemma} \label{lem:missconjclass} Let $G$ be a finitely generated nilpotent group and $H\leq G_f.$  Then there is a sequence of normal, finite index subgroups $(K_n)$ such that for all $x\in G$ if $x$ is not conjugate to an element of $H$, then there is an index $n_0$ so $n\geq n_0$ implies that $xK_n$ is not conjugate to an element of $HK_n.$
\end{lemma}
\begin{proof} From Lemma \ref{lem:containedinfc} obtain the finite set $F$ for $H.$ By Lemma \ref{thm:nilsep} applied to the groups $bHb^{-1}$ with $b\in F$ and then by Remark \ref{rem:moresep} there is a sequence $K_n$ such that $x\not\in \cup_{b\in F} bHb^{-1}$  implies
\begin{equation*}
xK_n\not\in \bigcup_{b\in F} bHb^{-1}K_n=\bigcup_{h\in H} \textup{conj}(hK_n).
\end{equation*}
where the $\supseteq$ inclusion comes from Lemma \ref{lem:containedinfc}.
\end{proof}
The proof of this section's main result relies on Kaniuth's induced traces for nilpotent groups \cite{Kaniuth06}.  We first recall the necessary details for us and refer the reader to Kaniuth's paper for the full story.  Let $G$ be a finitely generated nilpotent group and $H\leq G_f$ an abelian subgroup of the finite conjugacy subgroup of $G$. Let $\omega:H\rightarrow \T$ be a multiplicative character.   Define $\tilde \omega :G\rightarrow \C$ by $\tilde \omega (x)=\omega(x)$ if $x\in H$ and $\tilde \omega(x)=0$ otherwise. If $x\not\in G_f$ then we define $\text{Ind}_H^G \omega (x)=0.$  If $x\in G_f$, then the size of the conjugacy class is equal to the index $[G:C_G(x)]$ where $C_G(x)$ is the centralizer of $x.$ Let $A_x$ be a complete choice of coset representatives of $G/C_G(x).$  Then define
\begin{equation}\label{def:inducedtrace}
\text{Ind}_H^G\omega(x) = \frac{1}{[G:C_G(x)]}\sum_{a\in A_x} \tilde \omega(axa^{-1}).
\end{equation}
\begin{theorem}\label{thm:nilapprox} Every finitely generated nilpotent group is HS-stable.
\end{theorem}
\begin{proof} Let $G$ be a finitely generated nilpotent group and $\phi$ a character on $G.$
After modding out by the normal subgroup  $\{ g\in G:\phi(g)=1 \}$ we may assume that $\phi$ is faithful, i.e. that $\phi(g)=1$ implies $g=e.$ By \cite[Theorem 1.6]{Kaniuth06} there is a subgroup $H\leq G_f$  and a multiplicative character $\omega:H\rightarrow \T$ such that $\phi = \textup{Ind}_H^G \omega.$

Since $\phi$ is an extreme trace it restricts to a homomorphism on $Z(G)$. Therefore for $x\in Z(G)$ we have $\phi(x)\neq0.$ Hence $x$ is conjugate to an element of $H.$  But $x$ is only conjugate to itself so $x\in H$, i.e $Z(G)\leq H.$

Obtain a sequence $(K_n)$ for $H$ as in Lemma \ref{lem:missconjclass}.
By Lemma \ref{lem:approxhomo} there are homomorphisms $\omega_n:H\rightarrow \mathbb{T}$ with finite index kernels such that
\begin{equation*}
\lim_{n\rightarrow\infty}\omega_n(h)=\omega(h).
\end{equation*}
Replace $K_n$ by $K_n\bigcap \textup{ker} (\omega_n)$. By Remark \ref{rem:moresep} it still satisfies the conclusion of Lemma \ref{lem:missconjclass}.
Since $Z(G)$ has finite index in $G_f$ (\cite[Lemma 3]{Baer48}), it has finite index in $H$ and  it follows that $\textup{ker}(\omega_n|_{Z(G)})$ also has finite index in $Z(G)$.  Furthermore we have
\begin{equation}\label{eq1}
\bigcap_{n=1}^\infty (Z(G)\cap \textup{ker}(\omega_n))\subseteq Z(G)\cap \textup{ker}(\omega)=\{ e \}.
\end{equation}
Use Lemma \ref{lem:extendgroups} to obtain finite index normal subgroups $L_n\leq G$ with $L_n \cap Z(G)=\textup{ker}(\omega_n)\cap Z(G).$ By Remark \ref{rem:moresep}, the sequence $K_n\cap L_n$ still satisfies the conclusion of Lemma
\ref{lem:missconjclass} and by Lemma \ref{lem:trivint} and (\ref{eq1}) we have $\cap (K_n\cap L_n) = \{ e \}.$  Notice that since $K_n\cap L_n \cap H \leq \textup{ker}(\omega_n)$,  $\omega_n$ descends to a well defined homomorphism on $H/(K_n\cap L_n).$

Define $\phi_n:G\rightarrow \C$ by
\begin{equation*}
\phi_n(x) =\textup{Ind}_{H/(K_n\cap L_n)}^{G/(K_n\cap L_n)}\omega_n(x(K_n\cap L_n))
\end{equation*}
Suppose that $x$ is not conjugate to any element of $H.$  Then there is an $n_0$ large enough so $n\geq n_0$ implies $x(K_n\cap L_n)$ is not conjugate to any element of $H(K_n\cap L_n).$  Hence $0=\phi(x)=\phi_n(x).$ Next suppose $x$ is conjugate to an element of $H.$  Since both $\phi$ and $\phi_n$ are constant on conjugacy classes we may assume that $x\in H.$ Since $\cap (K_n\cap L_n)=\{e \}$ there is an $n_1\geq n_0$ so $n\geq n_1$ implies that $|\textup{conj}(x)|=|\textup{conj}(x(K_n\cap L_n))|.$  Since $\omega_n(x)\rightarrow \omega(x)$ we also have $\tilde \omega_n (axa^{-1})\rightarrow \tilde \omega(axa^{-1})$ for all $a\in G.$  It now follows from (\ref{def:inducedtrace}) that $\phi_n(x)\rightarrow \phi(x).$

Notice that the GNS representation associated with $\phi_n$ is finite dimensional, hence $G$ is HS-stable by Theorem \ref{thm:HSchar}. \end{proof}

\section{A few observations on MAP groups}\label{sec:MAP}

Since maximal almost periodicity (MAP) is a necessary condition for an amenable group to be HS-stable,
one needs some tools for proving/disproving MAP.  We provide a few observations on this topic including some easy-to-check obstructions to MAP that will be used in this and subsequent sections.

\subsection{Solvable MAP groups} The first three (or four)  lemmas are well known.

\begin{lemma}\label{Lemma1} Let $\Gamma$ be a Lie group and let $\Gamma_0$ be the connected component of the identity. Then
$\Gamma_0$ is an open normal subgroup of $\Gamma$.
\end{lemma}
\begin{proof} Any element of $\Gamma$ has a neighborhood homeomorphic to $\mathbb R^n$ and is therefore connected. Thus each element of $\Gamma_0$ has a neighborhood contained in $\Gamma_0$, so $\Gamma_0$ is open. Since $\Gamma_0 \times  \Gamma_0$ is connected, its image under the continuous map $\Gamma\times \Gamma \to \Gamma$, $(g, h) \mapsto gh$, is connected and contains the unit, hence is contained in $\Gamma_0$, so $\Gamma_0$ is a subgroup. It is clearly normal.\end{proof}

\begin{lemma}\label{Lemma2} If $\Gamma$ is a compact Lie group, then $\Gamma/\Gamma_0$ is finite.
\end{lemma}
\begin{proof} $\Gamma$ is a disjoint union of cosets, each of them is homeomorphic to $\Gamma_0$, hence open by Lemma \ref{Lemma1}. By compactness there can be only finitely many of them. \end{proof}

\begin{lemma}\label{Lemma3} If a subgroup of a topological group is open, then it is closed.
\end{lemma}
\begin{proof} If $H$ is an open subgroup of $G$, then $G\setminus H$,  being union of cosets which are homeomorphic to $H$ and hence are open, is open.\end{proof}

\begin{lemma}\label{Lemma4}  If $G$ is a solvable subgroup of a topological group $K$, then $\overline G$ is solvable.
\end{lemma}
\begin{proof}
 Since $G$ is solvable, there is a finite series $1 \triangleleft G_1 \triangleleft G_2 \triangleleft \ldots \triangleleft G$ with all $G_{i+1}/G_i$  abelian. Then we have a finite series
$1 \triangleleft \overline{G_1} \triangleleft \overline{G_2} \triangleleft \ldots \triangleleft \overline G$
and we only need to prove that $\overline{ G_{i+1}}/ \overline{ G_i}$  is abelian.
More generally, if $G/H$ is abelian, then $\overline G/ \overline H$ is abelian. Indeed, let $x, y \in \overline  G$. Then $x = \lim x_n, y= \lim y_n$, and $x_n y_n x_n^{-1}y_n^{-1}\in H$.  Then
$x y x^{-1}y^{-1} = \lim x_n y_n x_n^{-1}y_n^{-1}\in \overline H$, which means that
$\overline  G/\overline  H$ is abelian.
\end{proof}

\medskip

\begin{proposition}\label{abelian-by-RF} Let $G$ be a MAP solvable group. Then the intersection of all finite index normal subgroups of $G$ is abelian. In particular, $G$ is abelian-by-residually finite.
\end{proposition}
\begin{proof}
Recall the notation $N\lhd_f G$  means that $N$ is a normal finite index subgroup of $G$. Let
\begin{equation*}
H = \bigcap_{N\lhd_f G} N.
\end{equation*}
Then $G/H$ is residually finite.
We will show that $H$ is abelian. Let $\pi$ be a finite-dimensional representation of $G$, i.e. $\pi: G \to \mathcal U(n)$, for some $n\in \mathbb N$.
Let $\Gamma  = \overline{\pi(G)}$.  Then $\Gamma$ is compact and by Cartan's theorem (\cite[Theorem 34.7]{Zhelobenko06}),  it is a Lie group.
By Lemma \ref{Lemma2},  $\Gamma/\Gamma_0$ is finite. Then the kernel of the composition of $\pi$ with $\Gamma\to \Gamma/\Gamma_0$ is a finite index normal subgroup and hence contains $H$.   Therefore for any $h\in H$ we have
\begin{equation}\label{main}\pi(h) \in \Gamma_0.\end{equation}
Since  $G/\ker \;\pi$ is solvable, by Lemma \ref{Lemma4}  $\Gamma$ is solvable. Using Lemma \ref{Lemma1},  Lemma \ref{Lemma3} and again Cartan's theorem, we conclude that $\Gamma_0$ is a compact connected (Hausdorff) solvable group. By \cite[Corollary 29.5]{Zhelobenko06}  $\Gamma_0$ is abelian.
Then by (\ref{main}), for any $h_1, h_2\in H$, $\pi(h_1h_2h_1^{-1}h_2^{-1}) = 1.$ Since it holds for any finite-dimensional representation $\pi$ and since $G$ is MAP, for any $h_1, h_2\in H$ we have  $h_1h_2h_1^{-1}h_2^{-1} = 1,$ so $H$ is abelian.
\end{proof}

\begin{example} \label{ex:nonMAP}Proposition \ref{abelian-by-RF} provides an easy method to generate non MAP groups.
\begin{enumerate}

\item  $\mathbb H_3(\mathbb Q) = \left\{  \left( \begin{array}{ccc}  1 & x & z \\ 0 & 1 & y\\ 0 & 0 & 1 \end{array}  \right) : x,y,z\in \mathbb Q \right\}$ is not MAP.

Indeed, the subgroups $$\left\{  \left( \begin{array}{ccc}  1 & \ast & 0 \\ 0 & 1 & 0\\ 0 & 0 & 1 \end{array}  \right) \right\} \;\text{and}\; \left\{  \left( \begin{array}{ccc}  1 & 0 & 0 \\ 0 & 1 & \ast \\ 0 & 0 & 1 \end{array}  \right) \right\}$$
do not commute and are isomorphic to $\mathbb Q$. Then the intersection  of all normal subgroups of finite index contains both of them and therefore is not abelian.  By Proposition \ref{abelian-by-RF}, $\mathbb H_3(\mathbb Q)$ is not MAP. Since $\mathbb{H}_3(\Q)$ is a class 2 nilpotent group this example shows that finite generation in Theorem \ref{thm:nilapprox} is necessary.

\item  We generalize the above example.  Suppose $G$ and $H$ are solvable and do not admit non-trivial homomorphisms to finite groups. Then any non-trivial semi-direct product $G \rtimes H$ is not MAP. Indeed,  since $G \rtimes H$ is a non-trivial semi-direct product, there exist $g\in G$ and $h\in H$ such that $gh \neq hg$. Since the intersection  of all normal subgroups of finite index contains both $G$ and $H$, it is not abelian. Since $G \rtimes H$ is solvable, the statement follows from Proposition \ref{abelian-by-RF}. 
\end{enumerate}
\end{example}
\subsection{Another MAP obstruction}

\begin{proposition}\label{ref:MAPobstruction} If $G$ is MAP, then $[G, G] \bigcap Z(G)$ is residually finite.
\end{proposition}
\begin{proof}  Let $\pi$ be a finite-dimensional irreducible representation of $G$. Then $\pi$ sends elements of $Z(G)$ to scalar matrices and it sends elements of $[G, G]$  to matrices with determinant 1.
Thus elements of $[G, G] \bigcap Z(G)$ go to the set of scalar matrices with determinant 1 which is a finite group. Since $$\{\pi\;|_{[G, G] \bigcap Z(G)} \;|\; \pi \;\text{is a finite-dimensional irreducible representation of} \;G\}$$
separates points of $[G, G] \bigcap Z(G)$, $[G, G] \bigcap Z(G)$ is residually finite.
\end{proof}
Proposition \ref{ref:MAPobstruction} gives another proof that $\mathbb H_3(\mathbb Q)$ is not MAP. See Corollary \ref{lem:anothernonMAP} for another application.

\subsection{A characterization of certain MAP groups}

\begin{proposition}\label{prop:Mapchar} Let $H$ be amenable, $G$ abelian, and  $H$ act on $G$.
The following are equivalent:

\medskip

(i) $G \rtimes H$ is MAP,

\medskip

(ii) $H$ is MAP and, with respect to the induced action of $H $ on $\hat G$,  the set of characters with finite orbits is dense in $\hat G$.
\end{proposition}
\begin{proof} If $H$ is a discrete amenable group, $X$ a compact metric space and $\alpha$ an action of $H$ on $X$, then the crossed product C*-algebra
$C(X)\rtimes_{\alpha} H$ is residually finite-dimensional  (RFD) if and only if  $H$ is MAP and the union of finite orbits is dense in $X$  (\cite{Tomiyama87} and \cite{SkalskiShulman}). Thus $(ii)$ is equivalent to $C^*(G\rtimes H) = C(\hat G)\rtimes H$ being RFD. Since $H$ is amenable, by \cite{Bekka} the C*-algebra $C^*(G\rtimes H)$ being RFD is equivalent to $G\rtimes H$ being MAP. 
\end{proof}

\begin{example} We consider an example of a non-HS-stable semidirect product from \cite{Levit22}. Let $G = \mathbb Z(p^{\infty})$ be the Pr{\"u}fer $p$-group, $\Gamma= \mathbb Z$ and the action is the multiplication by powers of $q$, where $q$ is a prime distinct from $p$.
Then $\hat G$ is the group of $p$-adic integers and it is easy to check that the corresponding $\mathbb Z$-action on $\hat G$ does not have any periodic points. Therefore  $G \rtimes \Gamma$ is not MAP.
\end{example}

\section{Central extensions and HS-stable, permutation unstable groups}\label{sec:examples}
When searching for an amenable HS-stable non permutation stable group we were led naturally to consider the (in)permanence of HS-stability by central quotients.  We will eventually use central quotients to build our example but we start with the following Proposition which, as we will see, is the most general result possible for central quotients.

\begin{proposition}\label{prop:stableext} Suppose we have an extension $1 \to N \to G \to H \to 1$, where $N$ is finite and central and $G$ is an HS-stable group. Then $H$ is HS-stable.
\end{proposition}
\begin{proof} Since $N$ is finite it is a product of cyclic groups.  Arguing inductively we may without loss of generality assume that $N$ is cyclic.  For notational convenience we will moreover assume that $N=\Z/2\Z.$
  
 Let $\phi_n: H \to \mathcal U(l_n)$ be an approximate representation. Its composition with the quotient map $q: G\to H$ gives an approximate representation of $G$.
By HS-stability of $G$,  there are $l_n$-dimensional   representations $\rho_n$ of $G$ such that $\|\phi_n\circ q(g) - \rho_n(g)\|_2\to 0,$ for any $g\in G$.
Let $a$ be the generator of $N = \mathbb{Z}/2\mathbb{Z}$. For each $n$, there is an orthonormal basis of $\ell^2(l_n)$ such that $\rho_n(a)$ is of the form $\left(\begin{array}{cc} 1_{k_n} & \\ & -1_{m_n}\end{array}\right)$.  Since $N$ is central in $G$ and since the only matrices which commute with $\left(\begin{array}{cc} 1_{k_n} & \\ & -1_{m_n}\end{array}\right)$ are block-diagonal, $\rho_n = \rho_n^{(1)} \oplus  \rho_n^{(2)}$, where $\rho_n^{(1)}$ is a $k_n$-dimensional representation and $\rho_n^{(2)}$ is a $m_n$-dimensional representation. We have
$0 = \lim\|1_{l_n} - \rho_n(a)\|_2$,  which implies easily that
\begin{equation*}
1 =  \lim \text{tr}_{l_n} \rho_n(a) = \lim \frac{k_n-m_n}{k_n + m_n}
\end{equation*}
and we conclude that

\begin{equation}\label{dimension} \frac{m_n}{k_n} \to 0.
\end{equation}

Let $\chi$ be the trivial representation. The $l_n$-dimensional representation   $\pi_n =  \rho_n^{(1)} \oplus \chi^{(m_n)}$ is equal to $1_{l_n}$ on $N$ and therefore can be considered as a representation of $H$. It follows easily from (\ref{dimension}) that $\phi_n$ is $\|\|_2$-close to $\pi_n$.
 \end{proof}

Below we construct a two-step nilpotent group $G$ that is HS-stable but for any cyclic subgroup $N\leq Z(G)$ we have that $G/N$ is not a MAP group, and therefore not HS-stable. This example shows that the class of HS-stable groups is not closed under general central quotients and hence that Proposition \ref{prop:stableext} is the best result possible.  This gives another example (see Example \ref{ex:nonMAP}) of a non-MAP nilpotent group and showing (again) that the requirement of finite generation in Section \ref{sec:nilpotentcase} is necessary.
We will then give an example (defined by P. Hall in \cite{Hall61}) of a finitely generated cyclic-by-class 2 nilpotent HS-stable group. It was shown in \cite{Becker20} that this group is not permutation stable, hence this gives the first example of an amenable HS-stable group that is not permutation stable.
These two examples have much in common so we start with a common subsection to establish some results needed for both examples.

\subsection{Algebraic prerequisites}
Let $\ZH\leq \Q$ be the ring generated by $\frac{1}{2}.$ Everything in this section also works for $\Z[\tfrac{1}{p}]$ where $p$ is a prime, but for ease of notation we specify to $p=2.$

We first make several algebraic observations about $\ZH$ that are straightforward and most likely well-known. We include brief proofs for the convenience of the reader.
\begin{remark} \label{rem:algebra} The term \emph{square root} always refers to an additive square root.
\begin{enumerate}
\item Every element of $\ZH$ has a square root, hence every element of every quotient of  $\ZH$  also has a square root.
\item Let $k\in \Z$ be odd.  Then the map $\beta(x) =  x+x$ defines an automorphism of $\Z/k\Z.$ In particular every $x\in \Z/k\Z$ has a unique square root given by $\beta^{-1}(x)$.
\end{enumerate}
\end{remark}
\begin{lemma}\label{lem:RisRF} For every odd integer $k$ there is a ring homomorphism $\pi_k:\ZH\rightarrow \Z/k\Z$ determined by $\pi_k(1)=1$  such that
\begin{equation}\label{eq:morefaithful}
\pi_k\left(  \frac{a}{2^n} \right)\neq 0 \quad \text{ when } a\in\Z, 0<a<k, \hspace{.1in} n\geq0.
\end{equation}
\end{lemma}
\begin{proof} For every $n>1$ and $x\in \Z/k\Z$ let $\frac{x}{2^n}\in \Z/k\Z$ be the unique $2^n$th root of $x$ given by Remark \ref{rem:algebra}. Let $1\in \Z/k\Z$ denote a generator.  Then define $\pi_k:\ZH\rightarrow \Z/k\Z$ by $\pi_k(a/2^n)=a\frac{1}{2^n}$
where $a\in \Z$ and $n\geq 0.$ It is clear (by using uniqueness of square roots) that $\pi_k$ is a well-defined ring homomorphism.  Moreover notice that for each $n\geq0$, the element $\pi_k(1/2^n)\in \Z/k\Z$ is a group generator and hence has order $k.$   Thus we obtain (\ref{eq:morefaithful}).
\end{proof}
\begin{corollary} The group $\ZH$ is residually finite.
\end{corollary}
\begin{lemma}\label{lem:trivialhomo}   Let $A$ be a cyclic subgroup of $\ZH.$ Let $F$ be a finite group and $\pi:\ZH/A\rightarrow F$  a  homomorphism. Then $\pi$ is trivial. In particular $\ZH/A$ is not residually finite.
\end{lemma}
\begin{proof}  For convenience suppose that $A=\la 1 \ra.$
Without loss of generality suppose that $\pi$ is surjective, hence $F$ is abelian and by considering each factor of $F$ separately we may again without loss of generality suppose that $F$ is cyclic. Since each element of $\ZH/\Z$ has order a power of 2 so does every generator of $F$, i.e. $F\cong \Z/2^n\Z$ for some $n\geq 0.$ Since $\pi$ is surjective there is some $x\in \ZH/\Z$ such that $\pi(x)=1.$  Since $x$ has a square root by Remark \ref{rem:algebra} so does $\pi(x)=1.$  Since $1\in \Z/2^n\Z$ has a square root if and only if $n=0$ the homomorphism is trivial.
\end{proof}
Next we characterize the subgroups of $\ZH.$  Again, this is easy and most likely well-known but we were not able to locate a reference.
\begin{lemma}\label{lem:noncsub} Let $S\leq \ZH$ be non-cyclic.  Let $k_0=\min\{ k\in \Z^+: k\in S \}.$  Then $S=k_0\cdot\ZH = \la  \frac{k_0}{2^n}:n=0,1,... \ra.$  Consequently $\ZH/S\cong \Z/k_0\Z.$
\end{lemma}
\begin{proof} Let $S\leq \ZH$ and suppose that $S$ has a least positive element $x$. Since $S$ is not cyclic there is a $y\in S$ with $y>0$ such that $xk\neq y$ for all $k\in\Z^+.$  Hence there is a $k\in \Z^+$ such that $kx<y<(k+1)x.$  Then $0<y-kx<x$ and $y-kx\in S$ a contradiction.  We deduce that $S$ has no least positive element.
\\\\
For $n\geq 0$ define $x_n\in S$ and $k_n\in\Z^+$ such that
\begin{equation*}
x_n:=\frac{k_n}{2^n} = \min\{  x\in S\cap \la 2^{-n} \ra:x>0 \}.
\end{equation*}
Then $x_n\to 0 .$ Hence there is some index $n$ so $\frac{k_{n+1}}{2^{n+1}}<\frac{k_n}{2^n}.$  If $k_{n+1}=2k$ for some integer $k$, then $\frac{k_{n+1}}{2^{n+1}}=\frac{2k}{2^{n+1}}=\frac{k}{2^n}<x_n$, a contradiction. Therefore $k_{n+1}$ is odd.
For $0\leq i<n+1$ we have
\begin{equation*}
\frac{k_{n+1}}{2^i}=x_{n+1}2^{n+1-i}\in S\cap \la 2^{-i} \ra
\end{equation*}
Hence $k_i|k_{n+1}$ and in particular $k_i$ is odd.  Since $x_n\rightarrow 0$ we can find arbitrarily large pairs $x_n>x_{n+1}$ and then repeat the above argument to obtain $k_m$ is odd for all $m$. By definition we have $x_n\geq x_{n+1}$ for all $n.$ We claim this is strict for all $n.$  Indeed if there is an index $n$ so $x_n=\frac{k_n}{2^n} = \frac{k_{n+1}}{{2^{n+1}}}$, then $k_{n+1}$ is even, a contradiction.

We now show that $k_i=k_{i+1}$ for all $i\geq 0.$
By the previous paragraph we have $k_{i+1}=k_ij$ for some $j.$  Since $k_{i+1}$ is odd, $j$ is also odd.  Suppose that $j\geq3.$  Then
\begin{equation*}
0<\frac{k_i(j-2)}{2^{i+1}}=\frac{k_{i+1}}{2^{i+1}}-\frac{k_i}{2^i}\in S.
\end{equation*}
But $\frac{k_i(j-2)}{2^{i+1}}<\frac{k_{i+1}}{2^{i+1}}$, a contradiction.  Hence $j<3$, but since $j$ is odd we have $j=1.$
\\\\
This proves the first part of the lemma.  Consider $\pi_{k_0}:\ZH\rightarrow \Z/k_0\Z$ from Lemma \ref{lem:RisRF}.  One easily checks that $S=\ker(\pi_{k_0}).$
\end{proof}
\subsection{A class 2 nilpotent group}
\begin{definition} \label{def:Heis}
Let $R$ be a commutative unital ring.  Define
\begin{equation} \label{eq:Heisenberg}
\mathbb{H}_3(R)=\left\{  \left( \begin{array}{ccc}  1 & x & z \\ 0 & 1 & y\\ 0 & 0 & 1 \end{array}  \right) : x,y,z\in R \right\}.
\end{equation}
\end{definition}
\begin{remark} We record some elementary observations about $\mathbb{H}_3(R)$
\begin{enumerate}
\item $\mathbb{H}_3(R)$ is a class 2 nilpotent group and not finitely generated whenever $R$ is not a finitely generated abelian group.
\item The center of $\mathbb{H}_3(R)$, denoted $Z(\mathbb{H}_3(R))$ is isomorphic to $R$ and identified with those matrices satisfying $x=y=0$ in (\ref{eq:Heisenberg}).
\end{enumerate}
\end{remark}
\begin{corollary} $\HZH$ is residually finite.
\end{corollary}
\begin{proof} Apply the separating ring homomorphisms $\pi_k$ of Lemma \ref{eq:morefaithful} coordinate-wise to obtain separating  group homomorphisms $\tilde{\pi}_k: \HZH\rightarrow \mathbb{H}_3(\Z/k\Z).$
\end{proof}
By Proposition \ref{ref:MAPobstruction} we have
\begin{corollary} \label{lem:anothernonMAP}  Identify $\Z=\la 1 \ra\leq Z(\HZH).$ Then $\HZH/\Z$ is a class 2 nilpotent group that is not MAP.
\end{corollary}

We show that $\HZH$ is HS-stable.  Since $\HZH$ is not finitely generated, Theorem \ref{thm:nilapprox} does not apply.

\begin{theorem} \label{thm:Halltype} The group $\HZH$ is HS-stable.
\end{theorem}
\begin{proof} Let $\phi$ be an extreme trace on $\HZH$.  We show that $\phi$ can be approximated through matrices.  Note that $\phi$ restricts to a multiplicative character on $Z(\HZH).$  In particular $Z(\HZH)$ is in the multiplicative domain of $\phi$, i.e.
\begin{equation}\label{eq:multdom}
\phi(xz)=\phi(x)\phi(z) \text{ for all }x\in \mathbb{H}_3(\mathbb{Z}[\tfrac{1}{2}]), \text{ and }z\in Z(\mathbb{H}_3(\mathbb{Z}[\tfrac{1}{2}])).
\end{equation}
Let
\begin{equation*}
k(\phi)=\{ x\in \mathbb{H}_3(\mathbb{Z}[\tfrac{1}{2}]):\phi(x)=1 \}
\end{equation*}
Suppose first that $k(\phi)\cap Z(\HZH)$ is not cyclic.  It follows from Lemma \ref{lem:noncsub} and the fact that $\HZH$ is nilpotent of order 2  that, for an appropriate $k_0$ the subgroup $\mathbb{H}_3(k_0\ZH)/k(\phi)$ is abelian and (clearly) finite index in $\HZH/k(\phi).$  Hence we may consider $\phi$ as a trace on a virtually abelian group hence it can be approximated through matrices by \cite[Theorem 5]{Hadwin18}.
\\\\
Suppose then that $k(\phi)\cap Z(\HZH)=\la a \ra$ is cyclic.
Let $x\in \HZH\setminus Z(\HZH).$  We claim that $\phi(x)=0.$  Indeed one easily checks that there is a $y\in \HZH$ such that $[x,y]\in Z(\HZH)\setminus \la a \ra.$  We then have
\begin{equation*}
\phi(x)=\phi(y^{-1}xy)=\phi(x[x,y])=\phi(x)\phi([x,y])
\end{equation*}
where the last equality follows by (\ref{eq:multdom}).  Since $\phi([x,y])\neq 1$ we have $\phi(x)=0.$
It then follows from \cite[Theorem 7]{Hadwin18} that $\phi$ can be approximated through matrices.
\end{proof}
\begin{corollary}\label{2-stepNilpotent} For the central extension
\begin{equation*}
0\rightarrow \Z\rightarrow \mathbb{H}_3(\mathbb{Z}[\tfrac{1}{2}])\rightarrow \mathbb{H}_3(\mathbb{Z}[\tfrac{1}{2}])/\Z\rightarrow 0
\end{equation*}
the group $\HZH$ is HS-stable, but the quotient $\HZH/\Z$ is not MAP and therefore not HS-stable.  Hence Proposition
\ref{prop:stableext} can not be extended beyond finite groups.
\end{corollary}
\subsection{An HS-stable group that is not permutation stable}
We use an idea of P. Hall from \cite{Hall61}--in particular the group defined on Page 349--to build a finitely generated HS-stable group that is not permutation stable.  That the group is not permutation stable was shown by Becker, Lubotzky and Thom in \cite{Becker19}, our contribution is showing HS-stability.

\subsubsection{Topological dynamic preliminaries}   Let $\chi\in \dZH$, then $\chi$ is determined by the values $\chi(2^{-n})$ for $n\geq 0.$  Notice that $\chi(2^{-n})^2=\chi(2^{-n+1}).$  From this observation one easily checks that
\begin{equation}\label{eq:Rdualdesc}
\widehat{\mathbb{Z}[\tfrac{1}{2}]} \cong \{ (a_n)_{n=0}^\infty \in \T^\infty: a_{n+1}^2=a_n\text{ for }n\geq 0  \}
\end{equation}
where the isomorphism is given by $\chi\to (\chi(2^{-n}))_{n=0}^\infty.$\footnote{Alternatively notice that $\ZH$ is an inductive limit of $\Z$ via connecting homomorphisms $x\to 2x$ hence the dual group is the above projective limit of $\hat\Z\cong \T$}

Let $d:\T\times \T\rightarrow[0,1]$ be the metric $d(x,y)=\min\{ |\theta_1-\theta_2|: x=e^{2\pi i\theta_1}, y=e^{2\pi i\theta_2}  \}.$ Then define the metric $\rho$ on $\dZH$ by
\begin{equation*}
\rho(a,b)=\sum_{n\geq0} 2^{-n}d(a_n,b_n)
\end{equation*}
It is well known and easy to check that $\rho$ induces the usual topology of pointwise convergence on $\dZH.$ Let $a,b\in\dZH.$ An obvious but crucial remark is the following,
\begin{equation}\label{eq:coordest}
d(a_k,b_k)=d(a_{k+1}^2,b_{k+1}^2)\leq 2d(a_{k+1},b_{k+1}) \text{ for }k\geq0.
\end{equation}

Let $\alpha:\ZH\to \ZH$ be the automorphism $\alpha(x)=2x.$  Under the identification (\ref{eq:Rdualdesc}) the dual action and its inverse are given by
\begin{equation*}
\hat\alpha(a_0,a_1,a_2,\cdots)=(a_0^2,a_1^2,a_2^2,\cdots)=(a_0^2,a_0,a_1,\cdots)
\end{equation*}
\begin{equation} \label{eq:inverse}
\hat\alpha^{-1}(a_0,a_1,a_2,\cdots)=(a_1,a_2,a_3,\cdots).
\end{equation}
We investigate the periodic points of $\hat\alpha.$
\begin{lemma} \label{lem:periodic} Let $a\in\dZH$.  Then $a$ is an $N$-periodic point for $\hat\alpha$ if and only if $a_i$ is a $2^N-1$ root of unity for all $i=0,1,2,...$ Moreover for each $2^N-1$ root of unity $\lambda$ there is an $N$ periodic point $a$ with $a_{N-k}=\lambda^{2^k}$ for $k=0,...,N.$
\end{lemma}
\begin{proof} By (\ref{eq:inverse}) it is clear that each entry of a $N$ periodic point must be a $2^N-1$ root of unity.
Conversely, suppose each $a_i$ is a $2^N-1$ root of unity.  By the description of $\dZH$ in (\ref{eq:Rdualdesc}) we have
$a_{i+N}=a_{i+N}^{2^N}=a_i.$

For the final statement one repeats the already defined block $(a_0,a_1,...,a_{N-1})$ indefinitely and easily checks that this defines an element of $\dZH.$
\end{proof}

\begin{lemma}\label{lem:denseperiodic} Let $a\in \dZH.$  For any integer $N\geq 1$, there is a $2N+1$ $\hat\alpha$-periodic point $b\in\dZH$ such that $\rho(a,b)\leq \frac{7}{3}2^{-N}.$
\end{lemma}
\begin{proof} Let $\lambda\in\T$ be a $2^{2N+1}-1$ root of unity. By Remark \ref{rem:algebra} each $2^{2N+1}-1$ root of unity has a unique square root that is also a $2^{2N+1}-1$ root of unity.  Let $\lambda^{\frac{1}{2^N}}$ be the unique $2^N$th root of $\lambda$ that is also a $2^{2N+1}-1$ root of unity.  Define
\begin{equation*}
b_k= \left(  \lambda^{\frac{1}{2^N}} \right)^{2^{2N-k}}\text{ for }k=0,...,2N
\end{equation*}
Then repeat the initial block of $b$ indefinitely, which will create a $2N+1$ periodic point by Lemma \ref{lem:periodic}.

Now let $a\in\dZH$ be arbitrary.  Choose a $2^{2N+1}-1$ root of unity $\lambda$ such that $d(\lambda,a_N)<\frac{1}{2^{2N+1}-1}.$ Build the periodic point $b$ as above. Notice that $\lambda=b_N.$   We then have
\begin{align*}
\rho(a,b)& \leq \sum_{k=0}^N 2^{-k}d(a_k,b_k) + 2^{-N}\\
&\leq \sum_{k=0}^N 2^{-k}[2^{N-k}d(a_N,b_N)]+ 2^{-N}  \text{ (by (\ref{eq:coordest}))}\\
&\leq \sum_{k=0}^N 2^{-k}[2^{N-k}\frac{1}{2^{2N+1}-1}]+ 2^{-N} \\
&\leq \frac{7}{3}2^{-N.}
\end{align*}
\end{proof}
The conditions are quite strong on the next lemma but they provide a very simple proof and suffice for the purposes of this section.
\begin{lemma}\label{lem:Lipdense} Let $(X,d)$ be a compact metric space and $\gamma$ a homeomorphism such that $\gamma$ and $\gamma^{-1}$ are both Lipschitz with constant bounded by $L.$  Suppose that there is a $C>0$ such that for every $n\in\N$ there is a $N>n$ such that for every $x\in X$ there is an $N$-periodic point $y\in X$ such that $d(x,y)\leq CL^{-N/2}.$  Then the set of $\gamma$-invariant probability measures on $X$ with finite support is weak*-dense in the set of all $\gamma$-invariant probability measures.
\end{lemma}
\begin{proof}  Let $x\in X$ and obtain a large $N$-periodic point $y$ with $d(x,y)\leq CL^{-N/2}.$  For the notational convenience of the  next estimate suppose that $N=2k+1$ is odd (in the even case the summation goes from $-k$ to $k-1$).  By assumption we have
\begin{equation} \label{eq:estimate}
\sum_{i=-k}^{k} d(\gamma^i(x),\gamma^i(y))\leq d(x,y)\sum_{i=-k}^k L^{|i|}\leq CL^{-(2k+1/2)}\sum_{i=-k}^k L^{|i|}\leq \frac{2C\sqrt{L}}{L-1}.
\end{equation}
Let $\mathcal{C}$ be the collection of all $\gamma$-invariant probability measures with finite support. Notice that $\mathcal C$ is convex.  By way of contradiction suppose there is some $\gamma$-invariant probability measure $\nu$ that is not in the weak*-closure of $\mathcal C.$  Then by the Hahn-Banach separation theorem  there is a continuous function $f:X\rightarrow\R$ such that
\begin{equation} \label{eq:HBdiff}
r:=\sup_{\mu\in \mathcal C} \int f d\mu<\int f d\nu.
\end{equation}
Since Lipschitz functions are uniformly dense in $C(X)$ we may assume that $f$ is Lipschitz with constant $K$.  Also without loss of generality (by replacing $f$ with $f-\int f d\nu$) suppose that $\int f d\nu=0.$

By assumption there are arbitrarily large $N$ such that the $N$-periodic points in $X$ form a $CL^{-N/2}$-net.  Again for notational convenience suppose that $N=2k+1$ is odd.  Let
\begin{equation*}
g=\frac{1}{N}\sum_{i=-k}^k f\circ \gamma^i
\end{equation*}
Notice that for any $\gamma$-invariant measure $\mu$ we have $\int g d\mu=\int f d\mu$ and that $g$ is constant on each orbit of period $N.$  Let $x$ be an $N$ periodic point and $\mu = N^{-1}\sum_{i=-k}^k \delta_{\gamma^i(x)}$ which is clearly $\gamma$-invariant.  Then
\begin{equation*}
g(x)=\frac{1}{N}\sum_{i=-k}^k g(\gamma^i(x)) = \int gd\mu\leq r.
\end{equation*}
Let $y\in X$ be arbitrary and choose an $N$-periodic point $x$ with $d(x,y)\leq CL^{-\frac{N}{2}}.$ By inequality (\ref{eq:estimate}) we have
\begin{equation*}
|g(x)-g(y)|\leq \frac{1}{N}\sum_{i=-k}^k |f\circ\gamma^i(x)-f\circ\gamma^i(y)|\leq N^{-1}\frac{2KC\sqrt{L}}{L-1}
\end{equation*}
By taking $N$ large enough we then have that $g(y)\leq \frac{r}{2}$ for all $y\in X.$  Then $\int g d\nu\leq \frac{r}{2}<0$ a contradiction.
\end{proof}

\begin{corollary}\label{cor:denseperiodic} Let $\beta = \alpha\times \alpha^{-1}$ and consider the action  $\hat \beta=\hat\alpha\times \hat\alpha^{-1}$ on $\dZH\times \dZH.$  Then the $\hat\beta$-invariant probability measures with finite support are weak*-dense in the set of all $\hat\beta$-invariant probability measures.
\end{corollary}
\begin{proof} Notice that $\hat\alpha$ and $\hat\alpha^{-1}$ are both 2-Lipschitz. Hence the result follows from Lemmas \ref{lem:denseperiodic} and \ref{lem:Lipdense}.
\end{proof}
We record an obvious observation for later use
\begin{corollary}\label{cor:obviouscor} Let $k$ be an odd integer and let $\beta$ be as in Corollary \ref{cor:denseperiodic}. Notice that $\beta$ leaves the subgroup $(k\cdot \ZH)^2$ invariant.  Let $\pi:\ZH^2\rightarrow (k\cdot \ZH)^2$ be the isomorphism $(g,h)\mapsto (kg,kh).$  Then $\beta\circ\pi=\pi\circ\beta.$  In particular the $\hat \beta$-invariant measures on $(\widehat{k\cdot \ZH})^2$ with finite support are weak*-dense in the $\hat\beta$-invariant measures.
\end{corollary}
\subsubsection{Characterization of some HS-stable groups}
 We recall that an automorphism $\gamma$ of a measure space $(X,\mu)$ is called \textbf{essentially free} if
 \begin{equation*}
 \mu(\{ x\in X:\gamma^n(x)=x, \text{ for some }n\in\Z\setminus\{ 0 \} \})=0.
\end{equation*}
\begin{lemma} \label{lem:proper/periodic} Let $A$ be a unital $k$-homogeneous C*-algebra where $k\in \N$ and let $\gamma$ be an automorphism of $A.$   Let $\tau$ be an extreme trace on the crossed product $A\rtimes_\gamma\Z.$
Then $\tau|_A$ is a $\gamma$-invariant trace on $A.$  Let $\nu = \nu_{\tau|_A}$ be the $\hat\gamma$-invariant probability measure on $\hat A$ as in (\ref{eq:traceonhomo})
 Then
either  $\hat\gamma$ is a periodic automorphism of $(\hat A,\nu)$ or $\hat\gamma$  is essentially free.
\end{lemma}
\begin{proof}  Let $u\in A\rtimes_\gamma\Z$ be the unitary implementing $\gamma.$  Since $\tau$ is extreme $\pi_\tau(A\rtimes_\gamma\Z)''$ is a factor (Section \ref{sec:opalgpre}).
By (\ref{eq:traceonhomo}) we have $\pi_\tau(A)'' \cong L^\infty(\hat A,\nu)\otimes M_k.$

For each $n\in \Z$ let $F_n=\{ I\in\hat A:\hat\gamma^n(I)=I \}.$ Suppose that the action of $\hat\gamma$ on $(\hat A,\nu)$ is not essentially free. Then there is an $n\geq1$ such that $\nu(F_n)>0.$ Clearly $F_n$ is  $\hat\gamma$-invariant. Since $F_n$ is closed it is Borel and therefore the  function $1_{F_n}\otimes 1_k \in  L^\infty(\hat A,\nu)\otimes M_k\cong \pi_\tau(A)''.$  But since $F_n$ is $\hat\gamma$ invariant we have
\begin{equation*}
\pi_\tau(u)(1_{F_n}\otimes 1_k)\pi_\tau(u)^*=1_{\hat\gamma^{-1}(F_n)}\otimes 1_k=1_{F_n}\otimes 1_k.
\end{equation*}
Since $\pi_\tau(A\rtimes_\gamma\Z)''$ is generated by $\pi_\tau(A)$ and $\pi_\tau(u)$ it follows that $1_{F_n}\otimes 1_k$ is a central projection.  Moreover since $\nu(F_n)>0$ and $\pi_\tau(A\rtimes_\gamma\Z)''$ is a factor, we have $\nu(F_n)=1.$ But for all $I\in F_n$ we have $\hat\gamma^n(I)=I$, hence $\hat\gamma$ is a periodic automorphism of the measure space $(\hat A,\nu).$
\end{proof}

We will require the full generality of Lemma \ref{lem:proper/periodic} later, but in the case $k=1$ we combine the previous lemma with a classic result of Klaus Thomsen \cite[Theorem 4.3]{Thomsen95} to obtain the following corollary.  The same theorem in the finitely generated case was recently shown in \cite[Proposition 10.2]{Levit22}. In \cite{Levit22} the authors obtain a description of characters on metabelian groups and use it to obtain \cite[Proposition 10.2]{Levit22}.  We instead use Thomsen's theorem \cite[Theorem 4.3]{Thomsen95} on uniqueness of traces on crossed products.
\begin{corollary}\label{cor:densinvmeasure} Let $G$ be a countable abelian group and $\gamma$ an automorphism.  Then $G\rtimes_\gamma\Z$ is HS-stable if and only if the $\hat\gamma$- invariant measures on $\hat G$ with finite support are weak*-dense in the set of all $\hat\gamma$-invariant measures on $\hat G.$
\end{corollary}
\begin{proof} Suppose first that $G\rtimes_\gamma \Z$ is HS-stable and let $\mu$ be a $\hat\gamma$-invariant measure on $\hat G.$ Let $\tau$ be the trivial extension of $\mu$ to $G\rtimes_\gamma \Z.$
By Theorem \ref{thm:HSchar} let $\tau_n$ be a sequence of traces on $G\rtimes_\gamma \Z$ that factor through finite dimensional representations and converge to $\tau.$  Then $\tau_n$ restricted to $G$ determines a sequence of $\hat \gamma$-invariant measures on $\hat G$ with finite support that converge weak* to $\mu.$

Conversely, let $\tau$ be an extreme trace on $G\rtimes_\gamma \Z$ and let $\nu$ be the induced $\hat \gamma$-invariant measure on $\hat G.$  Suppose that $\hat\gamma$ is periodic on $(\hat G,\nu)$ with period $n$. Let $u\in C^*(G)\rtimes_\gamma\Z$ be the implementing unitary.  Since $\tau$ is extreme, $\pi_\tau(G\rtimes\Z)''$ is a factor and therefore $\pi_\tau(u^n)$ is a scalar multiple of the identity.  It follows that $\pi_\tau(G\rtimes_\gamma\Z)$ is virtually abelian and we may therefore regard $\tau$ as a trace on a virtually abelian group.  Hence $\tau$ can be approximated by matrices by \cite{Hadwin18}.

If the action is not periodic, then by Lemma \ref{lem:proper/periodic} the action is essentially free. By a result of Thomsen \cite{Thomsen95} (see \cite[Theorem 1.11]{Ursu21} for the formulation we are using) the $\gamma$-invariant trace $\nu$ on $C(\hat G)$ has a unique extension--which is necessarily $\tau$--to $C(\hat G)\rtimes_\gamma \Z.$  This unique extension is the trivial one and in particular we have $\tau(g)=0$ for all $g\not\in G.$

The rest of the proof is effectively contained in the definitions and observations of Section \ref{sec:opalgpre}.
By taking convex combinations we may, without loss of generality, suppose that $\nu$ can be weak*-approximated by measures supported on a single finite orbit.  Let $\mu$ be a measure supported on the orbit of a periodic point $I$.  Let $n$ be large so the period of $I$ divides $n.$  Let $\tau'$ be the trace on $G\rtimes_\gamma\Z$ obtained by the trivial extension.  Then the representations $\pi_{I,n}$ from
(\ref{eq:halfinduced}) and observation (\ref{eq:tpres}) show that\footnote{the key point here is that $\tau(g)=0$ when $g\not\in G$} $\tau'$ can be approximated through matrices.  It then follows that $\tau$ can also be approximated through matrices. Hence HS-stability follows from Theorem \ref{thm:HSchar}.
\end{proof}

\subsubsection{Main result}  We briefly recall the construction of twisted group C*-algebras from \cite{Packer92} for amenable groups\footnote{The usual reduced/full issues come into play for non-amenable groups, but all of the groups considered here are amenable.}.

Let $G$ be a discrete, amenable group and $\sigma:G\times G\to\T$ a 2-cocycle.  Then the \textbf{twisted group} C*-algebra $C^*(G,\sigma)$ is the universal C*-algebra generated by unitaries $\{ u_g:g\in G \}$ subject to the relations $u_gu_h=\sigma(g,h)u_{gh}.$  The twisted group C*-algebra $C^*(G,\sigma)$ is equipped with a canonical faithful trace $\tau$ satisfying $\tau(u_g)=0$ when $g\neq e.$

\begin{lemma}\label{lem:stronglyouter} Let $G$ be a countable amenable group and $\sigma:G\times G\rightarrow \T$ a 2-cocycle.  Let $\gamma$ be an automorphism of $G$ that preserves $\sigma$ and therefore induces an automorphism of $C^*(G,\sigma).$  Let $\tau$ be the canonical trace on $C^*(G,\sigma)$ and let $M=\pi_\tau(C^*(G,\sigma))''.$  Since $\tau\circ \gamma=\tau$ (trivially), $\gamma$ extends to an automorphism of $M.$  If for each $t\in G$ the set $\{ \gamma(g)^{-1}tg:g\in G \}$ is infinite, then $\gamma$ is an outer automorphism of $M.$
\end{lemma}
\begin{proof} Suppose to the contrary that there is a unitary $W\in M$ such that $Wu_g=u_{\gamma(g)}W$ for all $g\in G.$  Then consider the Fourier decomposition $W=\sum_{t\in G}c_tu_t$ where the sum is in the $L^2$-norm.  For every $g\in G$ we decompose the equation $Wu_g=u_{\gamma(g)}W$ to obtain
\begin{equation*}
\sum_{t\in G}\sigma(t,g)c_tu_{tg}= \sum_{t\in G}c_t\sigma(\gamma(g),t)u_{\gamma(g)t}=\sum_{t\in G}c_{\gamma(g)^{-1}tg}\sigma(\gamma(g),\gamma(g)^{-1}tg)u_{tg}
\end{equation*}
Hence for each $t$ we have $|c_t|=|c_{\gamma(g)^{-1}tg}|$ for all $g\in G.$  Since the set $\{ \gamma(g)^{-1}tg:g\in G \}$ is infinite if $c_t\neq0$ this would imply that $W$ has infinite $L^2$-norm a contradiction.  Hence $c_t=0$ implying that $W$ is not  unitary, a contradiction.
\end{proof}
\begin{definition}\label{def:indcocycle} Let $\chi \in \dZH.$  We define the 2-cocycle $\sigma_\chi:\ZH^2\times \ZH^2\rightarrow \T$ by $\sigma_\chi(g,h)=\chi(g_1h_2).$
\end{definition}
\begin{lemma}\label{lem:khomo} Let  $\chi \in \dZH $ with $\ker(\chi)=k\cdot \ZH$ for some odd integer $k.$  Let $\sigma=
\sigma_\chi$ as in Definition \ref{def:indcocycle}.  Let  $C^*(\ZH^2,\sigma)$ be the twisted group C*-algebra with canonical generators $\{ u_g:g\in \ZH^2 \}.$   Let $C=C^*(\{u_g:g\in k\cdot \ZH\times k\cdot \ZH\}).$ Then
\begin{enumerate}
\item $Z(C^*(\ZH^2,\sigma))=C$
\item The map $J\to J\cap C$ is a homeomorphism from $\text{Prim}(C^*(\ZH^2,\sigma))$ onto $\widehat{k\cdot \ZH}^2$
\item $C^*(\ZH^2,\sigma)$ is $k$-homogeneous with spectrum given in part 2 as $\widehat{k\cdot \ZH}^2.$
\end{enumerate}
\end{lemma}
\begin{proof} We first record a cocycle calculation.  Let $g,h\in \ZH^2.$ Then
\begin{equation*} 
u_gu_{g^{-1}} = \sigma(g,g^{-1})u_e, \text{ hence } u_g^{-1} = \sigma(g,g^{-1})^{-1}u_{g^{-1}}.
\end{equation*}
It then follows that
\begin{equation}\label{eq:unitarycomp}
u_gu_hu_g^{-1}=\sigma(g,h)\sigma(gh,g^{-1})\sigma(g,g^{-1})^{-1}u_h=\chi(g_1h_2-h_1g_2)u_h.
\end{equation}

Let $F=\{ (i,j)\in \Z^2:0\leq i,j\leq k-1 \}\subseteq \ZH^2.$ By Lemma \ref{lem:RisRF}  $F$ is a complete set of coset representatives for the quotient mapping $\ZH^2 \to \ZH^2/(k\cdot \ZH)^2 \cong (\Z/k\Z)^2.$ It follows that for any  $x\in C^*(\ZH^2,\sigma)$  there exist $\{ f_g:g\in F  \}\subseteq C$ such that
\begin{equation}\label{eq:uniquedecomp}
x = \sum_{g\in F}f_gu_g.
\end{equation}
Let $\tau$ be the canonical trace on $C^*(\ZH^2,\sigma).$  Then for any $f\in C$ and $(i,j)\in F\setminus \{ (0,0) \}$ we have $\tau(fu_{(i,j)})=0.$ Suppose now that $x=0$ has a decomposition as in (\ref{eq:uniquedecomp}).  Then
\begin{equation*}
0=\tau(x^*x)=\tau\left(  \sum_{g,h\in F}f_{h}^*f_gu_{h^{-1}g} \right)=\tau\left(   \sum_{g\in F} f_g^*f_g \right)
\end{equation*}
Since $\tau$ is faithful we have $f_g=0$ for all $g\in F.$ It follows that every element $x\in C^*(\ZH^2,\sigma)$ has a \emph{unique} decomposition as in (\ref{eq:uniquedecomp}).

Next suppose that $x\in Z(C^*(\ZH^2,\sigma))$ and that $x$ is decomposed as in (\ref{eq:uniquedecomp}).  Fix $(i,j)\in F\setminus\{ (0,0) \}.$  If $i=j$, set $(a,b) = (1,0)$ and if $i\neq j$ set $(a,b) = (1,1).$  In either case we have $k\not| aj-bi.$
 Then by comparing coefficients of $x$ and $x=u_{(a,b)}xu_{(a,b)}^{-1}$ we have by (\ref{eq:unitarycomp}) that
\begin{equation*}
f_{(i,j)}=\chi(aj-bi)f_{(i,j)}.
\end{equation*}
Since $k\not| aj-bi$, we have by Lemma \ref{lem:RisRF}  that  $\chi(aj-bi)\neq1$ hence $f_{(i,j)}=0.$  This proves assertion 1.

Let $c$ be a point in the spectrum of $C$ and consider the maximal ideal of $C$,  $I_c = \{ f\in C:f(c)=0 \}.$  Let $J_c$ be the ideal of $C^*(\ZH^2,\sigma)$ generated by $I_c.$  We will show that $C^*(\ZH^2,\sigma)/J_c\cong M_k$ which will simultaneously prove points 2 and 3.

First notice that
\begin{equation} \label{eq:defofJgam}
J_c = \left\{ \sum_{g\in F}f_gu_g: f_g\in I_c  \right\}
\end{equation}
Let $x$ be as in (\ref{eq:uniquedecomp}). We then have

\begin{equation*}
x+J_c = \sum_{g\in F}f_g(c)u_g + J_c
\end{equation*}
Hence
\begin{equation*}
C^*(\mathbb{Z}[\tfrac{1}{2}]^2,\sigma)/J_c = \text{span}\{ u_g+J_c: g\in F \}
\end{equation*}
We claim that the set $\{ u_g+J_c: g\in F \}$ is linearly independent.  To that end suppose that for some $\lambda_g\in\C$ we have
\begin{equation*}
\sum_{g\in F}\lambda_g u_g + J_c = 0+J_c.
\end{equation*}
By the description of $J_c$ in (\ref{eq:defofJgam}) it follows that $\lambda_g\in I_c$ for all $g$, i.e. that $\lambda_g=0.$ Hence the dimension of $C^*(\ZH^2,\sigma)/J_c$ is $k^2$. We now show that $C^*(\ZH^2,\sigma)/J_c$ has trivial center.

Let $x=\sum_{g\in F}\lambda_gu_g+J_c$ with $\lambda_g\in \C.$  If $x$ is central, then for all $h\in F$ we have by (\ref{eq:unitarycomp}) that
\begin{equation*}
\sum_{g\in F}\lambda_gu_g+J_c = \sum_{g\in F}\lambda_g\chi(g_1h_2-h_1g_2))u_g+J_c
\end{equation*}
As above, for each non-trivial $g$ one can find an $h\in F$ so $\chi(g_1h_2-h_1g_2)\neq1$, from which it follows that $x$ is a scalar multiple of the identity.
\end{proof}
\begin{theorem}\label{HS-stableNotPermutationStable} Define $\beta:\HZH\to\HZH$ by
\begin{equation*}
\beta\left( \left[ \begin{array}{ccc} 1 & x & z\\ 0 & 1 & y\\ 0 & 0 & 1 \end{array}  \right]  \right)=\left[ \begin{array}{ccc} 1 & 2x & z\\ 0 & 1 & \frac{y}{2}\\ 0 & 0 & 1 \end{array}  \right]
\end{equation*}
Then $\HZH\rtimes_\beta\Z$ is HS-stable. Moreover by the proof of \cite[Corollary 8.7]{Becker19}, the group $\HZH\rtimes_\beta\Z$ is not permutation stable.
\end{theorem}
\begin{proof} We show that $\HZH\rtimes_\beta\Z$ satisfies condition (iii) of Theorem \ref{thm:HSchar}.
Let $\tau$ be an extreme trace on $\HZH\rtimes_\beta\Z.$  Notice that $Z(\HZH\rtimes_\beta\Z)=Z(\HZH).$  Then $\tau$ restricted to $Z(\HZH)\cong \ZH$ is a character $\chi$.  Let $\sigma=\sigma_\chi$ be the cocycle defined in Definition \ref{def:indcocycle}. 

Define the lift $c:\ZH^2\rightarrow \HZH$ by
\begin{equation*}
c(x_1,x_2) = \left[ \begin{array}{ccc} 1 & x_1 & 0\\ 0 & 1 & x_2\\ 0 & 0 & 1 \end{array}  \right].
\end{equation*}
Then $\pi_\tau(c(x)c(y))=\sigma(x,y)\pi_\tau(c(xy)).$  By the universal property of twisted group C*-algebras the map $u_x\mapsto \pi_\tau(c(x))$ defines a quotient from $C^*(\ZH^2,\sigma)$ onto $\pi_\tau(C^*(\HZH)).$
\\\\
We now split the proof into two cases.
\newline
\emph{Case 1.}  The kernel of $\chi$ is cyclic. By the same reasoning as in the proof of Theorem \ref{thm:Halltype} we have that $\tau(g)=0$ for all $g\in \HZH\setminus Z(\HZH).$   Therefore $\pi_\tau(C^*(\HZH))\cong C^*( \ZH^2 ,\sigma).$  One checks that for every non-trivial $g\in \ZH^2$, there is an $h\in\ZH^2$ such that $\sigma(g,h)\overline{\sigma(h,g)}\neq 1.$  Hence there are no $\sigma$-regular conjugacy classes (\cite[Definition 1.1]{Packer89.1}) and therefore $C^*( \ZH^2 ,\sigma)$ is simple with unique trace by \cite[Theorem 1.7]{Packer89.1}.

By Lemma \ref{lem:stronglyouter} the action of $\beta$ on $C^*( \ZH^2 ,\sigma)$ extends to an outer automorphism of the von Neumann algebra generated by $C^*( \ZH^2 ,\sigma)$. Hence by \cite[Theorem 1]{Bedos93} the crossed product $C^*( \ZH^2 ,\sigma)\rtimes_\beta \Z$ is simple with a unique trace. Since $\pi_\tau(C^*(\HZH\rtimes_\beta\Z))$ is a quotient of $C^*( \ZH^2 ,\sigma)\rtimes_\beta \Z$ we must have that $\pi_\tau(C^*(\HZH\rtimes_\beta\Z))\cong C^*( \ZH^2 ,\sigma)\rtimes_\beta \Z$ also has a unique trace which is $\tau$.  This unique trace then satisfies $\tau(g)=0$ for $g\not\in Z(\HZH\rtimes_\beta\Z).$  Therefore $\tau$ is a centrally induced trace on a residually finite group and therefore can be approximated by matrices by \cite[Theorem 7]{Hadwin18}.
\newline
\emph{Case 2.} The kernel of $\chi$ is not cyclic.  Then by Lemma \ref{lem:noncsub} there is an odd integer $k$ so $\ker(\chi) = k\cdot \ZH.$ By Lemma \ref{lem:khomo}, $C^*(\ZH^2,\sigma)$ is $k$-homogeneous.  By Lemma \ref{lem:khomo} the spectrum of $C^*(\ZH^2,\sigma)$  is naturally identified with $(\widehat{k\cdot \ZH})^2.$

Let $\nu$ be the measure on  $X =\widehat{k\cdot\ZH}^2$ obtained by restricting $\tau$ as in Lemma \ref{lem:proper/periodic}.
By Lemma \ref{lem:proper/periodic} the action of $\hat\beta$ on $(X,\nu)$ is either periodic or essentially free.  If the action is periodic, then argue as in the proof of Corollary \ref{cor:densinvmeasure} that $\pi_\tau(\HZH\rtimes_\beta \Z)$ is a virtually abelian group and therefore  we may regard $\tau$ as a trace on a virtually abelian group.  By \cite[Theorems 4 and 5]{Hadwin18} it follows that $\tau$ is matricially approximable.

Finally suppose that the action of $\hat\beta$ on $(X,\nu)$ is essentially free. Then by \cite{Thomsen95} (again, see \cite[Theorem 11]{Ursu21} for the formulation we use), the measure $\nu$ has a unique tracial extension to the crossed product C*-algebra $C^*(\ZH^2,\sigma)\rtimes_\beta \Z.$ Then $\tau$ must be this extension and we have $\tau(g)=0$ for $g\not\in \HZH.$ Hence $\pi_\tau(C^*(\HZH\rtimes_\beta\Z))\cong C^*(\ZH^2,\sigma)\rtimes_\beta \Z.$

By Lemma \ref{cor:obviouscor} there is a sequence of $\beta$-invariant measures with finite support that converge weak* to $\nu.$ We now argue just as in the final paragraph of the proof of Corollary \ref{cor:densinvmeasure} to see that $\tau$ can be approximated by matrices.

In \cite[Corollary 8.7]{Becker19} they showed that a family of groups $A_p$ are not permutation stable.  The group $\HZH\rtimes_\beta\Z$ is not isomorphic to any $A_p$ nonetheless by replacing the group $H\leq A_p$ appearing in the proof of \cite[Corollary 8.7]{Becker19} with the pair $\{ \left[ \begin{array}{ccc} 1 & 0 & n\\ 0 & 1 & 0\\ 0 & 0 & 1 \end{array}  \right]:n\in \Z  \}\leq \HZH\rtimes_\beta\Z$ the proof of permutation instability of $\HZH\rtimes_\beta\Z$ follows nearly verbatim.
 \end{proof}

\section{(Very) flexible stability for amenable groups} \label{sec:flexiblestability}

For any Hilbert space $H$ with a fixed orthonormal basis $\{e_i\}$, the projection onto the span of the first $n$ basis vectors will be denoted by $P_n$. We identify $B(H)$ for $n$-dimensional $H$, with $M_n$.

\begin{definition} (Becker and Lyubotzky \cite{Becker20}) A C*-algebra $A$ is {\it flexibly HS-stable}  if for any approximate representation $\phi_n: A\to M_{k_n}$, $n\in \mathbb N$,  there exist $m_n\ge k_n$ with $\frac{k_n}{m_n}\to 1$ , and representations $\rho_n: A \to M_{m_n}$, $n\in \mathbb N$, such that for any $a\in A$
\begin{equation*}
\lim_{n\rightarrow\infty}\|\phi_n(a) - P_{k_n}\rho_n(a)P_{k_n}\|_2=0.
\end{equation*}
One obtains the notion of {\it very flexible stability} if one doesn't assume $\frac{k_n}{m_n}\to 1$. Similarly one defines these notions for a group $G.$
\end{definition}

\begin{proposition}\label{FlexiblyStable} Let $G$ be amenable. Then $G$ is flexibly HS-stable iff it is HS-stable.
\end{proposition}
\begin{proof} By the characterization of HS-stability for amenable groups, it is sufficient to prove that for a flexibly HS-stable amenable group $G$ any character $\chi$ can be approximated by traces of finite-dimensional representations of $G$. Since $G$ is amenable, $\chi$ is hyperlinear, so there is a free ultrafilter $\omega$ and an approximate representation $\phi_n: G \to \mathcal U(n)$, $n\in \mathbb N$, such that for any $g\in G$
\begin{equation}\label{1}
\lim_{n\rightarrow\omega}|\chi(g) - \text{tr}_n (\phi_n(g))|=0.\end{equation}
By flexible stability of $G$, there are  representations $\rho_n: A \to M_{m_n}$, $m_n\ge n$, $n\in \mathbb N$, such that $\frac{n}{m_n}\to 1$ and $\|\phi_n(g) - P_{n}\rho_n(g)P_{n}\|_2 \to 0, $  for any $g\in G$.
By Cauchy-Schwarz inequality
\begin{equation}\label{2}|\text{tr}_{n} (\phi_n(g)) - \text{tr}_{m_n} (P_{n}\rho_n(g)P_{n})| \le \|\phi_n(g) - P_{n}\rho_n(g)P_{n}\|_2 \to 0, \end{equation}  for any $g\in G$. Since
\begin{equation*}
\text{tr}_{m_n} (\rho_n(g)) = \frac{n \; \text{tr}_{n} (P_n\rho_n(g)P_n) + (m_n-n) \; \text{tr}_{m_n-n} ((1-P_n)\rho_n(g)(1-P_n))}{m_n},
\end{equation*}
 it follows easily that for any $g\in G$ we have
\begin{equation}\label{3}|\text{tr}_{m_n} (\rho_n(g)) - \text{tr}_n (P_n\rho_n(g)P_n)|\to 0.\end{equation}
By (\ref{1}), (\ref{2}),and  (\ref{3}) we have for any $g\in G$
\begin{equation*}
\chi(g) = \lim_{n\rightarrow\omega} \text{tr}_{m_n} \rho_n(g).
\end{equation*}
Hence $G$ is HS-stable.
\end{proof}

\begin{remark} It follows from the proof that for any group, flexible stability implies that any hyperlinear trace on $G$ can be approximated by traces of finite-dimensional representations.
\end{remark}

Recall that a unital C*-algebra $A$ has the {\bf  Lifting Property (LP)} if for any quotient C*-algebra $B/I$  any unital  comletely positive  (u.c.p.) map $\phi: A \to B/I$  lifts to a unital completely positive map from $A$ to $B$. For basic information on completely positive maps we refer to \cite[II.6.9]{Blackadar06}.

A C*-algebra is {\bf  residually finite-dimensional (RFD)} if for any $0\neq a\in A$ there is a finite-dimensional representation $\rho$ of $A$ such that $\rho(a)\neq 0.$

\begin{theorem}\label{LP+RFD} Every residually finite dimensional C*-algebra $A$ with the Lifting Property is very flexibly HS-stable.
\end{theorem}
\begin{proof} Let $\phi_n: A\to M_{k_n}$ be an approximate representation. It is sufficient to prove that for any finite  $F\subset A$ and any $\epsilon >0$, there is $n$ and a finite dimensional  representation $\rho^{(n)}$ such that $\|\phi_n(a) - P_{k_n}\rho^{(n)}(a)P_{k_n}\|_2\le \epsilon$, for any $a\in F$. The approximate representation $\phi_n$  corresponds to a $\ast$-homomorphism $\phi: A \to \prod M_{k_n}/I$, where $I$ is the ideal of all sequences that go to 0 w.r.t. $\|\|_2$-norm. Since $A$ has LP, $\phi$ lifts to a u.c.p. map $(\psi_n): A \to \prod M_{k_n}$. Hence there is $n$ such that
\begin{equation}\label{1} \|\phi_n(a) - \psi_n(a)\|_2 \le \epsilon/2,\end{equation}
for any $a\in F$.
By Stinespring's theorem, for each $n$ there is a representation $\pi_n: A\to B(H)$ such that
\begin{equation}\label{2}\psi_n = P_{k_n}\pi_n P_{k_n}.\end{equation} We may assume without loss of generality that all $\pi_n$'s act on an infinite-dimensional Hilbert space $H.$  We also fix an orthonormal basis $\{ e_n:n\geq1 \}$ for $H.$

Since $A$ is RFD, by  \cite{Hadwin14} for each $n$ there are finite rank  projections $Q_i^{(n)} \ge P_{k_n}$,  $i\in \mathbb N$, and representations $\rho_i^{(n)}: A \to B(Q_i^{(n)}H)$
such that
$$\pi_n(a) = SOT-\lim_{i} \rho_i^{(n)}(a)$$
 (here SOT stands for the strong operator topology).  In particular there is $i=i(n)$ such that for $Q^{(n)} = Q_{i(n)}^{(n)}$ and $\rho^{(n)} = \rho_{i(n)}^{(n)}$ we have
$$ \|(\pi_n(a) - \rho^{(n)}(a))e_k\|\le \epsilon/2,$$
for any $k\le n,$ $a\in F$.

Since for any matrix $T\in M_{k_n}$ we have
$$\|T\|_2 =  (\text{tr}_{k_n}(T^*T))^{1/2} = \left(\frac{\sum_{j=1}^{k_n} (T^*Te_j, e_j)}{k_n}\right)^{1/2} = \left(\frac{\sum_{j=1}^{k_n} \|Te_j\|^2}{k_n}\right)^{1/2},$$
we find that for each $n$  and any $a\in F$

\begin{multline}\label{3}\|P_{k_n}\pi_n(a)P_{k_n} - P_{k_n}\rho^{(n)}(a)P_{k_n}\|_2 = \left(\frac{\sum_{j=1}^{k_n} \|(P_{k_n}\pi_n(a)P_{k_n} - P_{k_n}\rho^{(n)}(a)P_{k_n})e_j\|^2}{k_n}\right)^{1/2}\\ \le \left(\frac{\sum_{j=1}^{k_n} \|(\pi_n(a) - \rho^{(n)}(a))e_j\|^2}{k_n}\right)^{1/2}
\le \epsilon/2.\end{multline}

From (\ref{1}), (\ref{2}), (\ref{3}) it follows that there is $n$ such that
$$\|\phi_n(a) - P_{k_n}\rho^{(n)}(a)P_{k_n}\|_2 \le \epsilon,$$ for any $a\in F$.

\end{proof}

\begin{remark} The assumption of the Lifting Property in Theorem \ref{LP+RFD} can be replaced by a weaker (at least formally) assumption of the Local Lifting Property (LLP).  Indeed, by \cite[Cor.1.7]{Ioana20} the assumption of $A$ being LLP is sufficient to lift completely positive maps from $A$ to a quotient of  the product of matrix algebras which is exactly where LP was used in the proof of Theorem \ref{LP+RFD}.
\end{remark}

\medskip

\begin{corollary} \label{VeryFlexiblyStable} Let $G$ be amenable. Then $G$ is very flexibly stable if and only if $G$ is MAP.
\end{corollary}
\begin{proof} Since $G$ is amenable, the C*-algebra $C^*(G)$ is nuclear and hence has LP  by Choi-Effros theorem \cite{ChoiEffros}.
Suppose $G$ is MAP. By \cite{Bekka}, an amenable MAP group has an RFD C*-algebra.  By Theorem \ref{LP+RFD}, $G$ is very flexibly stable.

Now suppose $G$ is very flexibly stable. Since $G$ is amenable,  by Proposition \ref{AmenableGroupsHyperlinear}  it embeds into the tracial ultraproduct of matrices. Hence there is an approximate homomorphism $\phi_n: G\to M_n$, such that for all $g\in G\setminus \{e\}$

\begin{equation*}
\lim_{n\rightarrow\infty}\|\phi_n(g)-1_n\|_2 \neq 0.
\end{equation*}
 By very flexible stability there are finite-dimensional representations $\rho_n$ such that for any $g\in G\setminus \{e\}$
\begin{equation*}
\|P_n\rho_n(g)P_n-1_n\|_2 \nrightarrow 0
\end{equation*}
In particular, for any $g\in G\setminus \{e\}$ there is $n$ such that $\rho_n(g) \neq 1.$

\end{proof}

\bibliographystyle{plain}

\end{document}